\documentclass[12pt]{amsart}
\usepackage{amssymb}
\usepackage{mathrsfs}
\usepackage{amssymb}
\usepackage{amsfonts}
\usepackage[english]{babel}
\usepackage[T1]{fontenc}
\usepackage[latin1]{inputenc}
\usepackage{fullpage}
\usepackage{amsmath}
\usepackage[all]{xy}
\usepackage{stmaryrd}

\newtheorem{theorem}{Theorem}[section]
\newtheorem{lemma}[theorem]{Lemma}
\newtheorem{corollary}[theorem]{Corollary}
\theoremstyle{definition}
\newtheorem{definition}[theorem]{Definition}

\newtheorem{proposition}[theorem]{Proposition}
\newtheorem{conjecture}[theorem]{Conjecture}

\theoremstyle{remark}
\newtheorem{remark}[theorem]{Remark}

\numberwithin{equation}{section}

\DeclareMathOperator{\Ima}{Im}

\begin{document}

\title{Deformations of Q-Curvarure I}


\author{Yueh-Ju Lin}
\address{(Yueh-Ju Lin) Department of Mathematics, University of Michgan, Ann Arbor, MI 48109, USA}
\email{yuehjul@umich.edu}

\author{Wei Yuan}
\address{(Wei Yuan) School of Mathematics and Computational Science, Sun Yat-sen University, Guangzhou, Guangdong 510275, China}
\email{gnr-x@163.com}

%    General info
%\subjclass[2000]{Primary 54C40, 14E20; Secondary 46E25, 20C20}

%\date{January 1, 2001 and, in revised form, June 22, 2001.}


\keywords{deformation of $Q$-curvature, linearized stability, local surjectivity, rigidity}

\thanks{Research of the second author supported by NSF grant DMS-1005295 and DMS-1303543.}

\begin{abstract}
In this article, we investigate deformation problems of $Q$-curvature on closed Riemannian manifolds. One of the most crucial notions we use is the \emph{$Q$-singular space}, which was introduced by Chang-Gursky-Yang during 1990's. Inspired by the early work of Fischer-Marsden, we derived several results about geometry related to $Q$-curvature. It includes classifications for nonnegative Einstein $Q$-singular spaces, linearized stability of non-$Q$-singular spaces and a local rigidity result for flat manifolds with nonnegative $Q$-curvature. As for global results, we showed that any smooth function can be realized as a $Q$-curvature on generic $Q$-flat manifolds, while on the contrary a locally conformally flat metric on $n$-tori with nonnegative $Q$-curvature has to be flat. In particular, there is no metric with nonnegative $Q$-curvature on $4$-tori unless it is flat.
\end{abstract}

\maketitle

%% The correct journal style for \specialsection is all uppercase; a known bug
%% in amsart.cls prevents this, so input must be uppercase until it is fixed.
%\specialsection*{This is a Special Section Head}

%%%%%%%%%%%%%%%%%%%%%%%%%%%%%%%%%%%%%%%%%%%%%%%%%%%%%%%%%%%%%%%%%%%%%%%%
%\footnote{Here is an example of a footnote. Notice that this footnote
%text is running on so that it can stand as an example of how a
%footnote with separate paragraphs should be written.
%\par
%And here is the beginning of the second paragraph.}%
%%%%%%%%%%%%%%%%%%%%%%%%%%%%%%%%%%%%%%%%%%%%%%%%%%%%%%%%%%%%%%%%%%%%%%%%

\section{Introduction}

In the theory of surfaces, the \emph{Gauss-Bonnet Theorem} is one of the most profound and fundamental results: 
\begin{align}\label{eqn:Gauss-Bonnet}
\int_M K_g dv_g = 2\pi \chi (M),
\end{align}
 for any closed surfaces $(M^2 ,g)$ with Euler characteristic $\chi (M)$. \\

On the other hand, the \emph{Uniformization Theorem} assures that any metric on $M^2$ is locally conformally flat. Thus an interesting question is that whether we can find a scalar type curvature quantity such that it generalizes the classic \emph{Gauss-Bonnet Theorem} in higher dimensions? \\

For any closed $4$-dimensional Riemannian manifold $(M^4, g)$, we can define the $Q$-curvature as follows
\begin{align}\label{Q_4}
Q_g = - \frac{1}{6} \Delta_g R_g - \frac{1}{2} |Ric_g|_{g}^2 + \frac{1}{6} R_g^2,
\end{align}
which satisfies the \emph{Gauss-Bonnet-Chern Formula} 
\begin{align}\label{Gauss_Bonnet_Chern}
\int_{M^4} \left( Q_g + \frac{1}{4} |W_g|^2_g \right) dv_g = 8\pi^2 \chi(M).
\end{align}
Here $R_g$, $Ric_g$ and $W_g$ are scalar curvature, Ricci curvature and Weyl tensor for $(M^4, g)$ respectively.\\

In particular, if $W_g = 0$, \emph{i.e.} $(M^4, g)$ is locally conformally flat, we have 
\begin{equation}\label{total_Q}
\int_{M^4} Q_g dv_g = 8\pi^2 \chi(M),
\end{equation}
which can be viewed as a generalization of (\ref{eqn:Gauss-Bonnet}).\\

Inspired by Paneitz's work (\cite{Paneitz}), Branson (\cite{Branson}) extended (\ref{Q_4}) and defined the $Q$-curvature for arbitrary dimension $n \geq 3$ to be 
\begin{align}
Q_{g} = A_n \Delta_{g} R_{g} + B_n |Ric_{g}|_{g}^2 + C_nR_{g}^2,
\end{align}
where $A_n = - \frac{1}{2(n-1)}$ , $B_n = - \frac{2}{(n-2)^2}$ and
$C_n = \frac{n^2(n-4) + 16 (n-1)}{8(n-1)^2(n-2)^2}$.\\

In the study of conformal geometry, there is an $4^{th}$-order differential operator closely related to $Q$-curvature, called Paneitz operator, which can be viewed as an analogue of conformal Laplacian operator:
\begin{align}
P_g = \Delta_g^2 - div_g \left[(a_n R_g g + b_n Ric_g) d\right] + \frac{n-4}{2}Q_g,
\end{align}
where $a_n = \frac{(n-2)^2 + 4}{2(n-1)(n-2)}$ and $b_n = - \frac{4}{n-2}$.\\

We leave the discussion of conformal covariance of $Q$-curvature and Paneitz operator in the appendix at the end of the article for readers who are interested in it.\\

The most fundamental motivation in this article is to seek the connection between $Q$-curvature and scalar curvature both as scalar-type curvature quantities. Intuitively, they should share some properties in common since both of them are generalizations of Gaussian curvature on surfaces. Of course, as objects in conformal geometry, a lot of successful researches have revealed their profound connections in the past decades. (See the appendix for a brief discussion.) However, when beyond conformal classes there are only very few researches on $Q$-curvature from the viewpoint of Riemannian geometry. %On the contrary, this is not the case for scalar curvature. Scalar curvature is always one of the central objects in the study of Riemannian geometry.\\
Motivated by the early work of Fischer and Marsden (\cite{F-M}) on the deformation of scalar curvature, we started to consider generic deformation problems of $Q$-curvature. \\

In order to study deformations of scalar curvature, the central idea of \cite{F-M} is to investigate the kernel of $L^2$-formal adjoint for the linearization of scalar curvature. To be precise, regard the scalar curvature $R(g)$ as a second-order nonlinear map on the space of all metrics on $M$, 
\begin{align*}
\notag
R: &\mathcal{M} \rightarrow C^{\infty}(M); \ 
g\mapsto R_{g}.
\end{align*}
Let $\gamma_g : S_2(M) \rightarrow C^\infty(M)$ be its linearization at $g$ and $\gamma_g^*: C^\infty(M) \rightarrow S_2(M)$ be the $L^2$-formal adjoint of $\gamma_g$, where $S_2(M)$ is the space of symmetric $2$-tensors on $M$. \\

A crucial related concept is the so-called \emph{vacuum static space}, which is the spatial slice of a type of special solutions to \emph{vacuum Einstein equations} (see \cite{Q-Y}). A vacuum static space can also be defined as a complete Riemannian manifold with $\ker \gamma_g^* \neq \{0\}$ (see the last section for an explicit definition). Typical examples of vacuum static spaces are space forms. In this sense, we can regard the notion of vacuum static spaces as a generalization of space forms. Of course, there are many other interesting examples of vacuum static spaces besides space forms, say $S^1\times S^2$ \emph{etc}. %%Need to give mor example
 The classification problem is a fundamental question in the study of vacuum static spaces, even in the field of mathematical general relativity. We refer the article \cite{Q-Y} for readers who are interested in it.\\

In fact, being vacuum static or not is the criterion to determine whether the scalar curvature is being linearized stable at this given metric. It was shown in \cite{F-M} that a closed non-vacuum static manifold is linearized stable and hence any smooth function sufficiently close to the scalar curvature of the background metric can be realized as the scalar curvature of a nearby metric. This result was generalized to non-vacuum static domains by Corvino (\cite{Corvino}). On the other hand, for a vacuum static space, rigidity results are expected. In \cite{F-M}, authors also showed that there is a local rigidity phenomenon on any torus. Later, Schoen-Yau and Gromov-Lawson (\cite{S-Y_1, S-Y_2, G-L_1, G-L_2}) showed that global rigidity also holds on torus. Inspired by the \emph{Positive Mass Theorem}, many people made a lot of progress in understanding the rigidity phenomena of vacuum static spaces (c.f. \cite{Min-Oo, A-D, A-C-G, Miao, S-T, B-M}). It was demonstrated that local rigidity is a universal phenomena in all vacuum static spaces (see \cite{Q-Y_1}).  \\

Along the same line, we can also consider $Q$-curvature as a $4^{th}$-order nonlinear map on the space of all metrics on $M$, 
\begin{align}
\notag
Q: &\mathcal{M} \rightarrow C^{\infty}(M); g\mapsto Q_g.
\end{align}
Due to the complexity of $Q$-curvature, this map is extremely difficult to study. However, by linearizing the map we may still expect some interesting results to hold. We denote $$\Gamma_g : S_2(M) \rightarrow C^\infty(M)$$ to be the linearization of $Q$-curvature at metric $g$. \\

%That is, a linear map from the tangent space of $\mathcal{M}$, which is the space of all symmetric $2$-tensors on $M$, to the space of smooth functions on $M$. \\

Now we give the following notion of $Q$-singular space introduced by Chang-Gursky-Yang in \cite{C-G-Y}, which plays a crucial role in this article.

\begin{definition}[Chang-Gursky-Yang \cite{C-G-Y}]
We say a complete Riemannian manifold $(M, g)$ is $Q$-singular, if $$\ker \Gamma_g^* \neq \{ 0 \},$$ where
$\Gamma_g^* : C^\infty(M) \rightarrow S_2(M)$ is the $L^2$-formal adjoint of $\Gamma_g$. We refer the triple $(M, g, f)$ as a $Q$-singular space as well, if $f (\not\equiv 0)$ is in the kernel of $\Gamma_g^*$.
\end{definition}

By direct calculations, we can obtain the precise expression of $\Gamma_g^*$ (see Proposition \ref{prop:Gamma^*}). Then follow their argument, we observed that the results in \cite{C-G-Y} for dimension $4$ can be extended to any other dimensions:

\begin{theorem}[Chang-Gursky-Yang \cite{C-G-Y}]\label{thm:Q_const}
%We have
%\begin{align}
%\Gamma_g^* f :=& A_n \left( - f \nabla^2 R - g \Delta^2 f + \nabla^2 \Delta f - Ric \Delta f + \frac{1}{2} g \delta (f dR) + \nabla ( f dR) \right)\\ \notag & - B_n \left( \Delta (f Ric) + 2 f \overset{\circ}{Rm}\cdot Ric + g \delta^2 (f Ric) + 2 \nabla \delta (f Ric) \right)\\ \notag &- 2 C_n \left( g\Delta (f R) - \nabla^2 (f R) + f R Ric \right)
%\end{align}
%and in particular,
%\begin{align}
%\mathscr{L}_g f := tr_g \Gamma_g^* f = \frac{1}{2} \left( P_g - \frac{n + 4}{2} Q_g \right) f. 
%\end{align}
A $Q$-singular space $(M^n, g)$ has constant $Q$-curvature and 
\begin{align}
\frac{n+4}{2}Q_g \in Spec(P_g).
\end{align}

\end{theorem}

As in the study of vacuum static spaces, the collection of all $Q$-singular spaces is also expected to be a ''small'' set in some sense, which lead us to the classification problems naturally. In fact, when restricted in the class of closed $Q$-singular Einstein manifolds with nonnegative scalar curvature, Ricci flat and spherical metrics are the only possible ones:

 \begin{theorem}\label{Classificaition_Q_singular_Einstein}
Suppose $(M^n,g,f)$ is a closed $Q$-singular Einstein manifold. If the scalar curvature $R_g \geq 0$, then
\begin{itemize}
\item $f$ is a non-vanishing constant if and only if $(M,g)$ is Ricci flat;
\item $f$ is not a constant if and only if $(M,g)$ is isometric to a round sphere with radius $r = \left( \frac{n(n-1)}{R_g}\right)^{\frac{1}{2}}$.
\end{itemize}
\end{theorem}

For non-$Q$-singular spaces, we can show that they are actually linearized stable, which answers locally prescribing $Q$-curvature problem for this category of manifolds:
\begin{theorem}\label{Q_stability}
Let $(M,\bar{g})$ be a closed Riemannian manifold. Assume
$(M,\bar{g})$ is not $Q$-singular, then the Q-curvature is linearized stable at $\bar g$ in the sense that $Q : \mathcal{M} \rightarrow C^{\infty}(M) $ is a submersion at $\bar{g}$.
Thus, there is a neighborhood $U \subset C^{\infty}(M)$ of $Q_{\bar{g}}$ such that for any $\psi \in U$, there exists a metric $g$ on $M$ closed to $\bar{g}$ with $Q_g = \psi$.
\end{theorem}

In particular, with the aid of Theorem \ref{Classificaition_Q_singular_Einstein}, we know a generic Einstein metric with positive scalar curvature has to be linearized stable:

\begin{corollary}\label{cor:stab_pos_Einstein}
Let $(M,\bar{g})$ be a closed positive Einstein manifold. Assume $(M,\bar{g})$ is not spherical, then the Q-curvature is linearized stable at $\bar g$.
\end{corollary}

Theorem \ref{Q_stability} actually provides an answer to global prescribing $Q$-curvature problem:
\begin{corollary}\label{cor:prescribing_zero_Q}
Suppose $(M, \bar g)$ is a closed non-$Q$-singular manifold with vanishing $Q$-curvature. Then any smooth function $\varphi$ can be realized as a $Q$-curvature for some metric $ g$ on $M$.
\end{corollary}

As a direct corollary of Theorem \ref{Q_stability}, we can obtain the existence of a negative $Q$-curvature metric as the following: 
\begin{corollary}\label{cor:pos_Y_negative_Q}
Let $M$ be a closed manifold with positive Yamabe invariant $Y(M) > 0$. There is a metric $g$ with $Q$-curvature $Q_g < 0$ on $M$. 
\end{corollary}

We also investigate the stability of Ricci flat metrics, which are in fact $Q$-singular. Due to the speciality of such metrics, we can give a sufficient condition for prescribing $Q$-curvature problem.
\begin{theorem}\label{Ricci_flat_stability}
Let $(M,\bar{g})$ be a closed Ricci flat Manifold. Denote $$\Phi:=\{\psi \in C^{\infty}(M): \int_M \psi dv_{\bar{g}} = 0\}$$ to be the set of smooth functions with zero average. Then for any $\psi \in \Phi$, there exists a metric $g$ on $M$ such that $$Q_g = \psi.$$

%In particular, let $(T^n,\bar{g})$ be the flat torus. Then for any smooth function $\psi$ with zero average, there exists a metric $g$ on $T^n$ such that $$Q_g = \psi.$$
\end{theorem}

On the other hand, we are interested in the question of how much the generic stability fails for flat metrics. Applying perturbation analysis, we discovered the positivity of $Q$-curvature is the obstruction for flat metrics exactly as what scalar curvature means for them. That is, flat metrics are locally rigid for nonnegative $Q$-curvature. As a special case, we proved a local rigidity result for $Q$-curvature on torus $T^n$ ($n \geq 3$), which is  similar to the non-existence result of positive scalar curvature metrics on torus due to Schoen-Yau and Gromov-Lawson (\cite{S-Y_1, S-Y_2, G-L_1, G-L_2}).

\begin{theorem}\label{flat_local_rigidity}
For $n \geq 3$, let $(M^n,\bar{g})$ be a closed flat Riemannian manifold and $g$ be a metric on $M$ with $$Q_g \geq 0.$$ Suppose $||g - \bar{g}||_{C^2(M, \bar{g})}$ is sufficiently small, then $g$ is also flat.
\end{theorem}

As for global rigidity, we have the following interesting result:

\begin{theorem}\label{thm:rigidity_tori}
Suppose $g$ is a locally conformally flat metric on $T^n$ with $$Q_g \geq 0,$$ then $g$ is flat. 

In particular, any metric $g$ on $T^4$ with nonnegative $Q$-curvature has to be flat.
\end{theorem}

It would be interesting to compare notions of $Q$-singular and vacuum static spaces. In fact, we have the following observation:

\begin{theorem}\label{Q-static_R-static}
Let $\mathcal{M}_R$ be the space of all closed vacuum static spaces and $\mathcal{M}_Q$ be the space of all closed $Q$-singular spaces. 
Suppose $(M, g, f) \in \mathcal{M}_R \cap \mathcal{M}_Q$, then $M$ is Einstein necessarily. In particular, it has to be either Ricci flat or isometric to a round sphere.
\end{theorem}

This article is organized as follows: We explain some general notations and conventions in Section 2. We then give a characterization of $Q$-singular spaces and proofs of Theorem \ref{thm:Q_const}, \ref{Classificaition_Q_singular_Einstein} in Section 3. In Section 4, we investigate several stability results, including Theorem \ref{Q_stability}, Corollary \ref{cor:stab_pos_Einstein}, \ref{cor:prescribing_zero_Q}, \ref{cor:pos_Y_negative_Q} and Theorem \ref{Ricci_flat_stability}. In Section 5, local rigidity of flat metrics, say Theorem \ref{flat_local_rigidity} and \ref{thm:rigidity_tori} will be shown. Some relevant results will also be discussed there. In Section 6, we discuss the relation between $Q$-singular and vacuum static spaces and prove Theorem \ref{Q-static_R-static} in the end of Section 6. In the end of this article, we provide a brief discussion about conformal properties of $Q$-curvature and Paneitz operator from the viewpoint of conformal geometry.\\

\paragraph{\textbf{Acknowledgement}}
The authors would like to express their appreciations to professor Sun-Yung Alice Chang, Professor Justin Corvino,  Professor Matthew Gursky, Professor Fengbo Hang, Professor Jie Qing and Professor Paul Yang, and Dr.Yi Fang for their interests in this work and inspiring discussions. Especially, we would like to thank Professor Matthew Gursky for pointing out the work in \cite{C-G-Y}. The authors are also grateful for MSRI/IAS/PCMI Geometric Analysis Summer School 2013, which created the opportunity for initiating this work. Part of the work was done when the second author visited \emph{Institut Henri Poincar\'e} and we would like to give our appreciations to the hospitality of IHP. \\

%%%%%%%%%%%%%%%%%%%%%%%%%%%%%%%%%%%%%%%%

%\section{Preliminaries}%%% section 2

\section{Notations and conventions}

Throughout this article, we will always assume $(M^n, g)$ to be an $n$-dimensional closed Riemannian manifold ($n \geq 3$) unless otherwise stated. Here by \emph{closed}, we mean compact without boundary. We also take $\{\partial_i\}_{i=1}^n$ to be a local coordinates around certain point. Hence we also use its components to stand for a tensor or vector in the article.\\

For convenience, we use following notions:

$\mathcal{M}$ - the set of all smooth metric on $M$;

$\mathscr{D}(M)$ - the set of all smooth diffeomorphisms $ \varphi : M \rightarrow M$;

$\mathscr{X}(M)$ - the set of all smooth vector fields on $M$;

$S_2(M)$ - the set of all smooth symmetric 2-tensors on $M$.\\

We adopt the following convention for Ricci curvature tensor,
\begin{align*}
R_{jk} = R^i_{ijk} = g^{il} R_{ijkl}.
\end{align*}

As for Laplacian operator, we simply use the common convention as follow,
\begin{align*}
\Delta_g := g^{ij} \nabla_i \nabla_j.
\end{align*}

Let $h, k \in S_2(M)$, for convenience, we define following operations:
\begin{align*}
(h \times k )_{ij} := g^{kl}h_{ik}k_{jl} = h_i^lk_{lj}
\end{align*}
and
\begin{align*}
h \cdot k  := tr (h \times k) = g^{ij}g^{kl}h_{ik}k_{jl} = h^{jk}k_{jk}.
\end{align*}

Let $X \in \mathscr{X}(M)$ and $h \in S_2(M)$, we use the following notations for operators
\begin{align*}
(\overset{\circ}{Rm} \cdot h )_{jk}:= R_{ijkl} h^{il},
\end{align*}
and
\begin{align*}
(\delta_g h)_i := - (div_g h)_i = -\nabla^j h_{ij},
\end{align*}
which is the $L^2$-formal adjoint of Lie derivative (up to scalar multiple) $$\frac{1}{2}(L_g X)_{ij} = \frac{1}{2} ( \nabla_i X_j + \nabla_j X_i).$$\\

%%%%%%%%%%%%%%%%%%%%%%%%%%%%%%%%%%%%%%%%%%%%%%

\section{Characterizations of Q-singular spaces}%%% section 3

Let $(M,g)$ be a Riemannian manifold and $\{g(t)\}_{t \in (-\varepsilon, \varepsilon)}$ be a 1-parameter family of metrics on $M$ with $g(0) =g$ and $g'(0) = h$. Easy to see, ${g'}^{ij} = - h^{ij}$.\\

We have the following well-known formulae for linearizations of geometric quantities. (c.f. \cite{C-L-N, F-M, Yuan})
\begin{proposition}
The linearization of Christoffel symbol is
\begin{align}
{\Gamma'}_{ij}^k = \frac{1}{2} g^{kl} \left( \nabla_i h_{jl} +\nabla_j h_{il} - \nabla_l h_{ij} \right).
\end{align}
\label{1st_variation_Ricci_scalar}
The linearization of Ricci tensor is
\begin{align}
Ric' = \left.\frac{d}{dt}\right|_{t=0}  Ric(g(t))_{jk} = - \frac{1}{2}\left( \Delta_L h_{jk} +
\nabla_j \nabla_k (tr h) + \nabla_j (\delta h)_k + \nabla_k (\delta
h)_j  \right),
\end{align}
where the Lichnerowicz Laplacian acting on $h$ is defined to be
\begin{align*}
\Delta_L h_{jk} = \Delta h_{jk} + 2 (\overset{\circ}{Rm}\cdot
h)_{jk} - R_{ji} h^i_k - R_{ki}h^i_j,
\end{align*}
and the linearization of scalar curvature is
\begin{align}\label{scalar_1st_variation}
R' = \left.\frac{d}{dt}\right|_{t=0}  R(g(t)) = - \Delta (tr h) +  \delta^2 h - Ric \cdot h.
\end{align}
\end{proposition}

For simplicity, we use $'$ to denote the differentiation with
respect to the parameter $t$ and evaluated at $t=0$.\\

\begin{lemma}
The linearization of the Laplacian acting on the scalar curvature is
\begin{align}\label{laplacian_scalar_variation}
(\Delta R)' = \left.\frac{d}{dt}\right|_{t=0}  \Delta_{g(t)} R(g(t)) = - \nabla^2 R\cdot h + \Delta
R' + \frac{1}{2} dR \cdot (d( tr h ) + 2\delta h).
\end{align}
\end{lemma}

\begin{proof}
Calculating in normal coordinates at an arbitrary point,
\begin{align*}
&\left.\frac{d}{dt}\right|_{t=0}  \Delta_{g(t)} R(g(t)) \\
=& \left.\frac{d}{dt} \right|_{t=0}\left({g^{ij}(\partial_i\partial_j R-\Gamma^k_{ij}\partial_kR)}\right)\\
=& -h^{ij}\partial_i\partial_jR+g^{ij}(\partial_i\partial_j R'-(\Gamma^k_{ij})'\partial_k R-\Gamma^k_{ij}\partial_k R')\\
=& -h^{ij}\partial_i\partial_j R - g^{ij}(\Gamma^k_{ij})'\partial_k R+\Delta R'\\
=& -h^{ij}\partial_i\partial_j R - \frac{1}{2}g^{ij}g^{kl}(\nabla_ih_{jl}+\nabla_jh_{il}-\nabla_lh_{ij})\partial_kR+\Delta R'\\
=& -h^{ij}\nabla_i\nabla_jR-\nabla^ih^k_i\nabla_kR+\frac{1}{2}\nabla^k(trh)\nabla_kR+\Delta R'\\
=& - \nabla^2 R\cdot h + \Delta R' + \frac{1}{2} dR \cdot (d( tr h ) + 2\delta h).
\end{align*}

\end{proof}

Now we can calculate the linearization of $Q$-curvature (1st variation).\\

\begin{proposition}\label{Q_1st_variation}
The linearization of $Q$-curvature is
\begin{align}\label{Q_linearization}
\Gamma_g h :=& DQ_g \cdot h 
= A_n \left( - \Delta^2 (tr h) +  \Delta \delta^2 h  -
\Delta ( Ric \cdot h ) + \frac{1}{2} dR \cdot (d( tr h ) + 2\delta h) - \nabla^2 R\cdot h\right)
\\ \notag & - B_n \left(  Ric \cdot \Delta_L h + Ric \cdot
\nabla^2(tr h) + 2 Ric \cdot\nabla (\delta h)   +2 (Ric\times Ric) \cdot h
\right) \\
\notag &+ 2 C_n R \left( - \Delta (tr h) +  \delta^2 h - Ric \cdot h
\right).
\end{align}
\end{proposition}

\begin{proof}
Note that
\begin{align*}
\left.\frac{d}{dt}\right|_{t=0} |Ric(g(t))|_{g(t)}^2
=&\left.\frac{d}{dt}\right|_{t=0} \left(g^{ik}g^{jl}R_{ij}R_{kl}\right)
=2g^{ik}g^{jl}R'_{ij}R_{kl}+2g'^{ik}g^{jl}R_{ij}R_{kl}\\
%=&2R'_{ij}R^{ij}-2h^{ik}g^{jl}R_{ij}R_{kl}\\
=&2Ric' \cdot Ric - 2(Ric\times Ric) \cdot h,
\end{align*}
then we finish the proof by combining it with (\ref{1st_variation_Ricci_scalar}),  (\ref{scalar_1st_variation}) and (\ref{laplacian_scalar_variation}).
\end{proof}

Now we can derive the expression of $\Gamma_g^*$ and hence the $Q$-singular equation $$\Gamma_g^* f = 0.$$
\begin{proposition}\label{prop:Gamma^*}
The $L^2$-formal adjoint of $\Gamma_g$ is
\begin{align*}
\Gamma_g^* f :=& A_n \left(  - g \Delta^2 f + \nabla^2
\Delta f - Ric \Delta f + \frac{1}{2} g \delta (f dR) + \nabla ( f
dR) - f \nabla^2 R  \right)\\ \notag
& - B_n \left( \Delta (f Ric) + 2 f
\overset{\circ}{Rm}\cdot Ric + g \delta^2 (f Ric) + 2 \nabla \delta (f
Ric) \right)\\ \notag
&- 2 C_n \left( g\Delta (f R) - \nabla^2 (f R) + f R
Ric \right).
\end{align*}
\end{proposition}

\begin{proof}
For any compactly supported symmetric 2-tensor $h \in S_2(M)$, we have
\begin{align*}
&\int_M f \left.\frac{d}{dt}\right|_{t=0}  {\Delta_{g(t)} R(g(t))}\ dv_g\\
=&\int_M f\left(- \nabla^2 R\cdot h + \frac{1}{2} dR \cdot (d( tr h ) + 2\delta h)+\Delta R'\right)dv_g\\
=&\int_M \left(-h \cdot f\nabla^2R+h \cdot \nabla(fdR) + \frac{1}{2}h \cdot \delta(fdR)g + \Delta f(-\Delta trh + \delta^2h-Ric\cdot h)\right)dv_g\\
=&\int_M  \left\langle-f\nabla^2R+\nabla(fdR) + \frac{1}{2}g\delta(fdR) -g\Delta^2f+\nabla^2\Delta f-Ric\Delta f,\ h\right\rangle dv_g.
\end{align*}\\

Similarly,
\begin{align*}
&\int_M f \left.\frac{d}{dt}\right|_{t=0} |Ric(g(t))|^2_{g(t)}\ dv_g\\
=&\int_M f\left(-2h^{il}g^{jk}R_{ij}R_{kl}+2R^{ij}{R'}_{ij}\right)dv_g\\
=&\int_M -2f(Ric\times Ric)\cdot h-fR^{ij}\left(\Delta_Lh_{ij} +\nabla_i\nabla_j(tr h)
+\nabla_i(\delta h)_j + \nabla_j(\delta h)_i\right) dv_g\\
=&\int_M \left\langle -\Delta_L (fRic)-2f(Ric\times Ric) - 2\nabla\delta(fRic) - g\delta^2(fRic),\ h \right\rangle dv_g\\
%=&\int_M \left\langle -\Delta (fRic)-2f\overset\circ{Rm}\cdot Ric+2f(Ric\times Ric) -2f(Ric\times Ric) - 2\nabla\delta(fRic) - g\delta^2(fRic),\ h \right\rangle dv_g\\
=&\int_M \left\langle-\Delta (fRic)-2f\overset\circ{Rm}\cdot Ric - 2\nabla\delta(fRic)-g\delta^2(fRic),\ h \right\rangle dv_g.
\end{align*}\\

And also,
\begin{align*}
&\int_M f \left.\frac{d}{dt}\right|_{t=0} (R(g(t)))^2 \ dv_g\\
=&\int_M 2 f R \left(-\Delta trh+\delta^2h-Ric\cdot h \right)dv_g\\
=&\int_M 2\left(-\Delta (fR)trh+\nabla^2(fR) \cdot h-(fR)Ric\cdot h\right)dv_g\\
=&\int_M 2 \left\langle -g\Delta (fR)+\nabla^2(fR)-(fR)Ric,\ h \right\rangle dv_g
\end{align*}

Combining all the equalities above, we have
\begin{align*}
\int_M f DQ_g \cdot hdv_g=\int_M \langle f,\Gamma_g h\rangle dv_g=\int_M \langle\Gamma_g^*f,h\rangle dv_g.
\end{align*}

\end{proof}

\begin{corollary}
\begin{align*}
\mathscr{L}_g f := tr_g \Gamma_g^* f = \frac{1}{2} \left( P_g - \frac{n + 4}{2} Q_g \right) f. 
\end{align*}
\end{corollary}

\begin{proof}
Taking trace of $\Gamma_g^* f$, we have
\begin{align*}
tr \Gamma_g^* f =& -2 Q_g f - (n-1) A_n \Delta^2 f - \left( \frac{n-4}{2} A_n + \frac{n}{2}B_n + 2(n-1)C_n \right)f \Delta R \\
&- (A_n + B_n + 2(n-1)C_n) R \Delta f - \left( \frac{n-2}{2}A_n + n B_n + 4(n-1) C_n\right) df \cdot dR \\
&- (n-2) B_n Ric \cdot \nabla^2 f\\
=&  -2 Q_g f + \frac{1}{2} \Delta^2 f - \frac{(n-2)^2 + 4}{4(n-1)(n-2)} R \Delta f - \frac{n-6}{4(n-1)} df \cdot dR + \frac{2}{n-2} Ric \cdot \nabla^2 f\\
=&  -2 Q_g f + \frac{1}{2} \left( \Delta^2 f - a_n R \Delta f - \left(a_n + \frac{1}{2}b_n\right) df \cdot dR - b_n Ric \cdot \nabla^2 f \right)\\
=& \frac{1}{2} \left( \Delta_g^2 f- div_g \left[(a_n R_g g + b_n Ric_g) df \right] \right) - 2 Q_g \\
= &\frac{1}{2} \left( P_g - \frac{n+4}{2} Q_g \right) f. 
\end{align*}

Here we used the fact that $\frac{n-4}{2} A_n + \frac{n}{2}B_n + 2(n-1)C_n =0$.
\end{proof}

Combining above calculations, we can justify that $Q$-curvature is a constant for $Q$-singular spaces using exactly the same argument in \cite{C-G-Y}. For the convenience of readers, we include a sketch of the proof as follows. For more details, please refer to \cite{C-G-Y}.

\begin{theorem}[Chang-Gursky-Yang \cite{C-G-Y}]%\label{thm:Q_const}
A $Q$-singular space $(M^n, g)$ has constant $Q$-curvature and 
\begin{align}
\frac{n+4}{2}Q_g \in Spec(P_g).
\end{align}

\end{theorem}
%%%%%%%%insert Statement of Thm 1.2
\begin{proof}%[Proof of Theorem \ref{thm:Q_const}] 
We only need to show $Q_g$ is a constant.

For any smooth vector field $X \in \mathscr{X}$ on $M$, we have 
\begin{align*}
\int_M \langle X,   \delta_g \Gamma_g^* f \rangle dv_g = \frac{1}{2}\int_M \langle L_X g, \Gamma_g^* f \rangle dv_g = \frac{1}{2}\int_M f\ \Gamma_g (L_X g)\ dv_g  = \frac{1}{2}\int_M \langle f dQ_g, X \rangle \ dv_g.
\end{align*}
Thus $$fdQ_g = 2 \delta_g \Gamma_g^* f = 0$$ on $M$. 

Suppose there is an $x_0 \in M$ with $f(x_0) = 0$ and $dQ_g (x_0) \neq 0$. By taking derivatives, we can see $f$ vanishes to infinite order at $x_0$. \emph{i.e.} $$\nabla^m f(x_0) = 0$$ for any $m \geq 1$. Since $f$ also satisfies that $$\mathscr{L}_g f = \frac{1}{2} \left(P_g f -  \frac{n+4}{2}Q_g f\right) = 0,$$ by applying the Carleman estimates of Aronszajn (\cite{A}), we can conclude that $f$ vanishes identically on $M$. But this contradicts to the fact that $g$ is $Q$-singular. 

Therefore, $dQ_g$ vanishes on $M$ and thus $Q_g$ is a constant. Here we obtain the proof of Theorem \ref{thm:Q_const}.
\end{proof}

\textbf{Examples of $Q$-singular spaces.}

\begin{itemize}
  \item
  Ricci flat spaces\\
  Take $f$ to be a nonzero constant, then it satisfies the $Q$-singular equation.\\

  \item
  Spheres\\
  Take $f$ to be the $(n+1)^{th}$-coordinate function $x_{n+1}$ restricted on $$S^n = \{x\in \mathbb{R}^{n+1}: |x|^2 = 1\}.$$ Then $f$ satisfies the Hessian type equation $$\nabla^2 f + g f = 0.$$ With the aid of the following Lemma \ref{Q-static_Einstein}, we can easily check that $f$ satisfies the $Q$-singular equation.\\

  \item
  Hyperbolic spaces\\
  Similarly, if we still take $f$ to be the $(n+1)^{th}$-coordinate function $x_{n+1}$ but restricted on $$H^n = \{(x',x_{n+1})\in \mathbb{R}^{n+1}: |x'|^2 - |x_{n+1}|^2 = -1, x' \in \mathbb{R}^n, x_{n+1} > 0\}.$$ Then $f$ satisfies the similar Hessian type equation $$\nabla^2 f - g f = 0,$$ and hence solves the $Q$-singular equation.
\end{itemize}

Based on the complexity of the $Q$-singular equation, one can easily see that it is very difficult to study the geometry of generic $Q$-singular spaces. However, when restricted to some special classes of Riemannian manifolds we can still get some interesting results of $Q$-singular spaces.

\begin{lemma}\label{Q-static_Einstein}
Let $(M,g)$ be an Einstein manifold and $f \in \ker \Gamma_g^*$, we have
\begin{align}
\Gamma_g^* f = A_n \left[ \nabla^2 \left(\Delta f + \Lambda_n R f
\right) + \frac{R}{n(n-1)}g \left(\Delta f + \Lambda_n R f\right)
\right] = 0,
\end{align}
where $\Lambda_n : = \frac{2}{A_n} (\frac{B_n}{n} + C_n) = - \frac{(n+2)(n-2)}{2n(n-1)} < 0$, for any $n \geq 3$.
\end{lemma}

\begin{proof}

Since $g$ is Einstein, \emph{i.e.} $Ric=\frac{R}{n}g$, we get
\begin{align*}
\Gamma_g^* f=&A_n \left[ -g\Delta^2f+\nabla^2\Delta f-\frac{R}{n}g\Delta f \right] -B_n\left[\frac{R}{n}\Delta (fg)+2f \overset\circ{Rm}\cdot \frac{R}{n}g + 2\frac{R}{n}\nabla\delta(fg) +\frac{R}{n}g\delta^2(fg) \right]\\
& \ \ -2C_n\left[Rg\Delta f-R\nabla^2f+\frac{R^2}{n}fg\right]\\
=&A_n \left[-g\Delta^2f+\nabla^2\Delta f-\frac{R}{n}g\Delta f\right] - 2B_n \left[ \frac{R}{n}g\Delta f+\frac{R^2}{n^2}gf - \frac{R}{n}\nabla^2f \right] \\
& \ \ -2C_n \left[ Rg\Delta f-R\nabla^2f+\frac{R^2}{n}fg \right]\\
=& A_n \left[-g\Delta^2f+\nabla^2\Delta f - \left(\frac{1}{n} + \Lambda_n\right) R g\Delta f + \Lambda_n R \nabla^2 f - \frac{\Lambda_n}{n}R^2 g f \right]\\
=& A_n \left[ \nabla^2 \left(\Delta f + \Lambda_n R f\right)  - g \left( \Delta \left(\Delta f + \Lambda_n R f\right) +\frac{R}{n}\left(\Delta f + \Lambda_n R f\right) \right) \right].
\end{align*}

By assuming $f\in \ker \Gamma_g^*$, \emph{i.e.} $\Gamma_g^*(f)=0$, when taking trace, we have
\begin{align*}
A_n \left[  - (n-1) \Delta \left(\Delta f + \Lambda_n R f\right) - R\left(\Delta f + \Lambda_n R f\right) \right] = 0.
\end{align*}

Thus,
\begin{align*}
\Delta \left(\Delta f + \Lambda_n R f\right) = - \frac{R}{n-1}\left(\Delta f + \Lambda_n R f\right).
\end{align*}

Now substitute it in the expression of $\Gamma_g^* f$, we have
$$
\Gamma_g^* f =A_n \left[ \nabla^2 \left(\Delta f + \Lambda_n R f\right) + \frac{R}{n(n-1)}g \left(\Delta f + \Lambda_n R f\right)\right] = 0,
$$
where $\Lambda_n= \frac{2}{A_n}\left(\frac{B_n}{n}+C_n\right) = - \frac{(n+2)(n-2)}{2n(n-1)} < 0$, for any $n \geq 3$.\\
\end{proof}

Next we show that a closed $Q$-singular Einstein manifold with positive scalar curvature has to be spherical:
\begin{proposition} \label{Q-static_sphere}
Let $(M^n,g)$ be a complete $Q$-singular Einstein manifold with positive scalar curvature. Then $(M^n,g)$ is isometric to the round sphere $(S^n(r), g_{_{S^n(r)}})$, with radius $r = \left(\frac{n(n-1)}{R_g}\right)^{\frac{1}{2}}$. Moreover, $\ker \Gamma_g^*$ is consisted of eigenfunctions of $(-\Delta_g)$ associated to $\lambda_1 = \frac{R_g}{n-1}> 0$ on $S^n(r)$ and hence $\dim \ker \Gamma_{g}^* = n+1$.
\end{proposition}

\begin{proof}

Note that $\Lambda_n < 0$ and $Spec (-\Delta_g)$ consisted of nonnegative real numbers, then $$\Lambda_n R_g \not\in Spec ( - \Delta_g).$$

Let $f \in \ker \Gamma_g^*$, $f\not\equiv 0$ and $\varphi := \Delta_g f + \Lambda_n R_g f$, then
\begin{align*}
\varphi \not\equiv 0.
\end{align*}

Therefore, by Lemma \ref{Q-static_Einstein}, we have the following equation
\begin{align}\label{Obata_type_eqn}
\nabla^2\varphi=-\frac{R_g}{n(n-1)}\varphi g
\end{align}
with a nontrivial solution $\varphi$.\\

Taking trace, we get
\begin{align}
\Delta_g \varphi = - \frac{R_g}{n-1}\varphi.
\end{align}
\emph{i.e.} $\frac{R_g}{n-1}$ is an eigenvalue of $-\Delta_g$.\\

On the other hand, by \emph{Lichnerowicz-Obata's Theorem}, the first nonzero eigenvalue of $(- \Delta_g)$ satisfies
$$\lambda_1 \geq \frac{R_g}{n-1}$$
with equality if and only if $(M^n,g)$ is isometric to the round sphere $(S^n(r), g_{S^n(r)})$ with radius $r = \left(\frac{n(n-1)}{R} \right)^{\frac{1}{2}}$. Hence the first part of the theorem follows.\\

Since $\Lambda_n R_g \not\in Spec ( - \Delta_g)$, it implies that the operator $\Delta_g + \Lambda_n R_g$ is invertible. On the other hand,
\begin{align*}
\left( \Delta_g + \Lambda_n R_g \right) \left(\Delta_g + \frac{R_g}{n-1}\right) f = \left(\Delta_g + \frac{R_g}{n-1}\right) \left( \Delta_g + \Lambda_n R_g \right) f = \left(\Delta_g + \frac{R_g}{n-1}\right) \varphi=0.
\end{align*}
Thus $$\left(\Delta_g + \frac{R_g}{n-1}\right) f = 0.$$ \emph{i.e.} $f$ is an eigenfunction associated to $\lambda_1 = \frac{R_g}{n-1}$.\\

Therefore, $\ker \Gamma_g^*$ can be identified by the eigenspace associated to the first nonzero eigenvalue $\lambda_1> 0$ of $(-\Delta_g)$ on $S^n(r)$ and hence $\dim\Gamma_g^* = n+1$.

\end{proof}

As for Ricci flat ones, we have:

\begin{theorem}\label{Static_Ricci_flat}
Suppose $(M^n,g)$ is a $Q$-singular Riemannian manifold. If $(M,g)$ admits an nonzero constant potential $f \in \ker \Gamma_g^*$, then $(M,g)$ is $Q$-flat, \emph{i.e.} the $Q$-curvature vanishes identically. Furthermore, suppose $(M,g)$ is a closed $Q$-singular Einstein manifold, then $(M,g)$ is Ricci flat if and only if $\ker \Gamma_g^*$ is consisted of constant functions.
\end{theorem}

\begin{proof}

Without lose of generality, we can assume that $1 \in \ker \Gamma_g^*$. We have
\begin{align*}
tr\Gamma^{*}_{g} 1 = \mathscr{L}_g 1  = \frac{1}{2} \left( P_g 1 - \frac{n+4}{2}Q_g\right) =  - 2 Q_g = 0,
\end{align*}
which implies that $Q_g=0$.\\

Now assume $(M,g)$ is a closed $Q$-singular Einstein manifold.\\

We have
\begin{align*}
\int_M Q_g dv_g &= \int_M \left(A_n \Delta_{g} R + B_n |Ric|^2 + C_n R^2\right) dv_g \\
&= \left(\frac{B_n}{n} + C_n \right) R^2\ Vol_g(M) \\
&= \frac{(n+2)(n-2)}{8n(n-1)^2}R^2\ Vol_g(M).
\end{align*}

Suppose $1 \in \ker \Gamma_g^*$, then $Q_g = 0$ and hence 
\begin{align*}
R = 0 ,
\end{align*}
That means $(M, g)$ is Ricci flat.\\

On the other hand, if $(M,g, f)$ is Ricci flat, 
\begin{align*}
\mathscr{L}_g f = \frac{1}{2}\Delta^2 f = 0.
\end{align*}

Since $M$ is compact without boundary, $f$ is a nonzero constant on $M$.

\end{proof}

\begin{remark}
It was shown in \cite{C-G-Y} that for a closed $4$-dimensional $Q$-singular space $(M, g)$, $1 \in \ker \Gamma_g^*$ if and only if $g$ is Bach flat with vanishing $Q$-curvature. But this theorem doesn't generalized to other dimensions automatically. In fact, for generic dimensions, $1 \in \ker \Gamma_g^*$ doesn't imply $g$ is Bach flat by direct calculations, although it is still $Q$-flat as shown above. 
\end{remark}

Theorem \ref{Classificaition_Q_singular_Einstein} follows from combining Propositions \ref{Q-static_sphere} and \ref{Static_Ricci_flat}.\\

For Einstein manifolds with negative scalar curvature, generically they are not $Q$-singular:
\begin{proposition}\label{Q-hyperbolic}
Let $(M, g)$ be a closed Einstein manifold with scalar curvature $R < 0$. Suppose $\Lambda_n R \not\in Spec(-\Delta_g)$, then $$\ker \Gamma_g^* = \{0\}.$$ \emph{i.e.} $(M, g)$ is not $Q$-singular.
\end{proposition}

\begin{proof}
For any smooth function $f\in  \ker \Gamma_g^*$, let $\varphi := \Delta f + \Lambda_n R f$. By Lemma \ref{Q-static_Einstein}, we have $$\nabla^2 \varphi + \frac{R}{n(n-1)}g\varphi = 0.$$

Taking trace,
$$\Delta \varphi + \frac{R}{n-1}\varphi = 0.$$
Thus $$\varphi = \Delta f + \Lambda_n R f = 0$$ identically on $M$, since $0 > \frac{R}{n-1} \not\in spec(-\Delta_g)$. Thus $$\Delta f = - \Lambda_n R f .$$ 

By assuming $\Lambda_n R \not\in Spec(-\Delta_g)$, we conclude that $f \equiv 0$. That means, $\ker \Gamma_g^*$ is trivial.

\end{proof}

%%%%%%%%%%%%%%%%%%%%%%%%%%%%%%%%%%%%%%%%%

\section{Stability of Q-curvature}

In this section, we will discuss the linearized stability of $Q$-curvature on closed manifolds.
As main tools, we need following key results. Proofs can be found in referred articles.

\begin{theorem}[Splitting Theorem (\cite{B-E, F-M})]\label{Splitting_Theorem}
Let $(M,g)$ be a closed Riemannian manifold, $E$ and $F$ be vector bundles on $M$. Let $D : C^\infty (E) \rightarrow C^\infty(F)$ be a $k^{th}$-order differential operator and $D^* : C^\infty(F) \rightarrow C^\infty(E)$ be its $L^2$-formal adjoint operator.

For $k \leq s\leq \infty$ and $1< p < \infty$, let $D_s : W^{s,p}(E) \rightarrow W^{s-k,p}(F)$ and $D_s^* : W^{s,p}(F) \rightarrow W^{s-k,p}(E)$ be the bounded linear operators by extending $D$ and $D^*$ respectively.

Assume that $D$ or $D^*$ has injective principal symbols, then
\begin{align}
W^{s,p}(F) =  \Ima D_{s+k} \oplus \ker D_s^*.
\end{align}
Moreover,
\begin{align}
C^\infty(F) = \Ima D \oplus \ker D^*.
\end{align}
\end{theorem}

In particular, take the vector bundle $F$ to be all symmetric 2-tensors $S_2(M)$, we have

\begin{corollary}[Canonical decomposition of $S_2$ (\cite{B-E, F-M})]\label{Canonical_decomposition_S_2}
Let $(M,g)$ be a closed Riemannian manifold, then the space of symmetric 2-tensors can be decomposed into
\begin{align}
S_2(M) = \{ L_X g : X \in \mathscr{X}(M)\} \oplus \ker \delta_g.
\end{align}
\end{corollary}

Another result we need is

\begin{theorem}[Generalized Inverse Function Theorem (\cite{Gel'man})]\label{Generalized_Inverse_Function_Theorem}
Let $X$, $Y$ be Banach spaces and $f : U_{x_0} \rightarrow Y$ be a continuously differentiable map with $f(x_0) = y_0$, where $U_{x_0} \subset X$ is a neighborhood of $x_0$. Suppose the derivative $D f (x_0) : X \rightarrow Y$ is a surjective bounded linear map. Then there exists $V_{y_0} \subset Y$, a neighborhood of $y_0$, and a continuous map $\varphi : V_{y_0} \rightarrow U_{x_0}$ such that
\begin{align*}
f(\varphi(y)) = y, \ \ \ \forall y\in V_{y_0};
\end{align*}
and
\begin{align*}
\varphi(y_0) = x_0.
\end{align*}

\end{theorem}

Combining Corollary \ref{Canonical_decomposition_S_2} and Theorem \ref{Generalized_Inverse_Function_Theorem}, we obtain \emph{Ebin's Slice Theorem}(\cite{Ebin}), which we will use later in this article.

\begin{theorem}[Slice Theorem (\cite{Ebin, F-M})]\label{Slice_Theorem}
Let $(M,\bar{g})$ be a Riemannian manifold. For $p> n$, suppose that $g$ is also a Riemannian metric on $M$ and $||g - \bar{g}||_{W^{2,p}(M,\bar{g})}$ is sufficiently small, then there exists a diffeomorphism $\varphi \in \mathscr{D} (M)$ such that $h := \varphi^* g - \bar{g}$ satisfies that $\delta_{\bar{g}} h = 0$ and moreover,
$$||h||_{W^{2,p}(M,\bar{g})} \leq N ||g - \bar{g}||_{W^{2,p}(M,\bar{g})},$$
where $N$ is a positive constant only depends on $(M, \bar g)$.
\end{theorem}

\begin{remark}
Brendle and Marques (c.f. \cite{B-M}) proved an analogous decomposition and slice theorem for a compact domain with boundary.
\end{remark}

Now we can give the proof of the main theorem (Theorem \ref{Q_stability}) in this section.
\begin{theorem}%\label{Q_stability}
Let $(M,\bar{g})$ be a closed Riemannian manifold. Assume
$(M,\bar{g})$ is not $Q$-singular, then the Q-curvature is linearized stable at $\bar g$ in the sense that $Q : \mathcal{M} \rightarrow C^{\infty}(M) $ is a submersion at $\bar{g}$.
Thus, there is a neighborhood $U \subset C^{\infty}(M)$ of $Q_{\bar{g}}$ such that for any $\psi \in U$, there exists a metric $g$ on $M$ closed to $\bar{g}$ with $Q_g = \psi$.
\end{theorem}
%%%%%%%insert statement of Thm 1.4
\begin{proof}%[Proof of Theorem \ref{Q_stability}]

The principal symbol of $\Gamma_{\bar{g}}^*$ is
\begin{align*}
\sigma_{\xi}(\Gamma_{\bar{g}}^*) = - A_n \left( g |\xi|^2 - \xi \otimes\xi \right)|\xi|^2.
\end{align*}

Taking trace, we get
\begin{align*}
tr\ \sigma_{\xi}(\Gamma_{\bar{g}}^*) = - A_n \left( n - 1 \right)|\xi|^4.
\end{align*}

Thus, $\sigma_{\xi}(\Gamma_{\bar{g}}^*) = 0$ will imply that $\xi = 0$. \emph{i.e.} $\Gamma^*_{\bar{g}}$ has an injective principal symbol.

By the \emph{Splitting Theorem} \ref{Splitting_Theorem},
$$C^\infty(M) = \Ima \Gamma_{\bar{g}} \oplus \ker \Gamma^*_{\bar{g}},$$
which implies that $\Gamma_{\bar{g}}$ is surjective, since we assume that $(M,\bar{g})$ is not $Q$-singular, \emph{i.e.} $\ker \Gamma^*_{\bar{g}} = \{0\}$.

Therefore, applying the \emph{Generalized Implicit Function Theorem} (Theorem \ref{Generalized_Inverse_Function_Theorem}), $Q$ maps a neighborhood of $\bar g$ to a neighborhood of $Q_{\bar{g}}$ in $C^\infty(M)$.

\end{proof}

As a consequence, we can derive the stability of generic positive Einstein manifolds (Corollary \ref{cor:stab_pos_Einstein}).

\begin{corollary}%\label{cor:stab_pos_Einstein}
Let $(M,\bar{g})$ be a closed positive Einstein manifold. Assume $(M,\bar{g})$ is not spherical, then the Q-curvature is linearized stable at $\bar g$.
\end{corollary}

%%%%%%insert statement of corollary 1.5
\begin{proof}%[Proof of Corollary \ref{cor:stab_pos_Einstein}]

Since $M$ is assumed to be not spherical, by Theorem \ref{Classificaition_Q_singular_Einstein}, $(M,g)$ is not $Q$-singular. Now the stability follows from Theorem \ref{Q_stability}. \\
\end{proof}

For a generic $Q$-flat manifold, we can prescribe any smooth function such that it is the $Q$-curvature for some metric on the manifold (Corollary \ref{cor:prescribing_zero_Q}).

\begin{corollary}%\label{cor:prescribing_zero_Q}
Suppose $(M, \bar g)$ is a closed non-$Q$-singular manifold with vanishing $Q$-curvature. Then any smooth function $\varphi$ can be realized as a $Q$-curvature for some metric $ g$ on $M$.
\end{corollary}
%%%%%%%insert statement of cor 1.6
\begin{proof}%[Proof of Corollary \ref{cor:prescribing_zero_Q}]
Since $(M, \bar g)$ is non-$Q$-singular, applying Theorem \ref{Q_stability}, as a nonlinear map, $Q$ maps a neighborhood of the metric $\bar g$ to a neighborhood of $Q_{\bar g} = 0$ in $C^\infty(M)$. Thus there exists an $\varepsilon_0 > 0$ such that for any smooth function $\psi$ with $||\psi||_{C^\infty(M)}< \varepsilon_0$, we can find a smooth metric $g_\psi$ closed to $\bar g$ with $Q_{g_\psi} = \psi$.

Now for any nontrivial $\varphi \in C^\infty(M)$, let $\tilde \varphi:= u_{\varepsilon_0 , \varphi}\varphi$, where $u_{\varepsilon_0 , \varphi} = \frac{\varepsilon_0 }{2 ||\varphi||_{_{C^\infty(M)}}} > 0$. Clearly, $||\tilde \varphi||_{C^\infty(M)} < \varepsilon_0$ and hence there is a metric $g_{\tilde \varphi}$ closed to $\bar g$ such that $Q_{g_{\tilde \varphi}} = \tilde \varphi$. Let 
$$g = u_{\varepsilon_0 , \varphi}^{\frac{1}{2}} g_{\tilde \varphi},$$ then we have $$Q_g = u_{\varepsilon_0 , \varphi}^{-1} \cdot Q_{g_{\tilde \varphi}} =  \varphi.$$
\end{proof}

Now we can prove the Corollary \ref{cor:pos_Y_negative_Q}.
\begin{corollary}
Let $M$ be a closed manifold with positive Yamabe invariant $Y(M) > 0$. There is a metric $g$ with $Q$-curvature $Q_g < 0$ on $M$. 
\end{corollary}
\begin{proof}%%%add reference of Matsuo
By Matsuo's theorem (see Corollary 2 in \cite{Mat14}), on a closed manifold $M$ with dimension $n\geq 3$ and positive Yamabe invariant, there exists a metric $g$ with scalar curvature $R_g = 0$ but $Ric_g \not\equiv 0$ on $M$. Thus the $Q$-curvature satisfies $$Q_g = - \frac{2}{(n-2)^2} |Ric_g|^2 \leq 0.$$
If $|Ric_g|>0$ pointwisely, then $Q_g < 0$ on $M$. Otherwise, there is a point $p \in M$ such that $|Ric_g(p)|^2 = 0$ and hence $Q_g$ is not a constant on $M$. This implies that the metric $g$ is not $Q$-singular. Therefore, by Theorem \ref{Q_stability}, we can perturb the metric $g$ to obtain a metric with strictly negative $Q$-curvature. This gives the conclusion.
\end{proof}

For the Ricci flat case, we have a better result since we can identify $\ker \Gamma_{\bar{g}}^*$ with constants.

\begin{theorem}%\label{Ricci_flat_stability}
Let $(M,\bar{g})$ be a closed Ricci flat Manifold. Denote $$\Phi:=\{\psi \in C^{\infty}(M): \int_M \psi dv_{\bar{g}} = 0\}$$ to be the set of smooth functions with zero average. Then for any $\psi \in \Phi$, there exists a metric $g$ on $M$ such that $$Q_g = \psi.$$
\end{theorem}
%%%%insert Thm 1.8
\begin{proof}%[Proof of Theorem \ref{Ricci_flat_stability}]
In the proof of Theorem \ref{Q_stability}, we have showed that the principal symbol of $\Gamma_{\bar{g}}^*$ is injective and the decomposition
$$C^\infty(M) = \Ima \Gamma_{\bar{g}} \oplus \ker \Gamma^*_{\bar{g}}$$
holds.

On the other hand, by Theorem \ref{Static_Ricci_flat}, $\ker \Gamma^*_{\bar{g}}$ is consisted by constant functions and hence $ \Ima \Gamma_{\bar{g}} = \Phi$. By identifying $\Phi$ with its tangent space, we can see the map $Q$ is a submersion at $g$ with respect to $\Phi$. Therefore, by \emph{Generalized Inverse Function Theorem} (Theorem \ref{Generalized_Inverse_Function_Theorem}), we have the local surjectivity. That is, there exists a neighborhood of $\bar{g}$, say $U_{\bar{g}} \subset \mathcal{M}$ and a neighborhood of $0$, say $V_0 \subset C^\infty(M)$ such that $ V_0|_\Phi \subset Q_{g}(U_{\bar{g}})$, where $g\mapsto Q_{g}$ is a map.\\

Now for any $\psi \in \Phi$, let $r > 0$ be a sufficiently large constant such that $\frac{1}{r^4}\psi \in V_0$. There exists a metric $g_r \in U_{\bar{g}}$, such that $Q_{g_r} = \frac{1}{r^4} \psi$. Let $g := \frac{1}{r^2}g_r$, we have $$Q_g = r^4 \cdot Q_{g_r} = r^4 \cdot \frac{1}{r^4}\psi = \psi.$$
Thus we prove Theorem \ref{Ricci_flat_stability}.
\end{proof}

Similarly, we can also give an answer to the prescribing $Q$-curvature problem near the standard spherical metric by noticing the fact that 
$$\ker \Gamma_{\bar{g}}^* = E_{\lambda_1}$$ from Theorem \ref{Q-static_sphere}:

\begin{theorem}\label{thm:stability_sphere}
Let $(S^n, \bar{g})$ be the standard unit sphere and $E_{\lambda_1}$ be the eigenspace of $(-\Delta_g)$ associated to the first nonzero eigenvalue $\lambda_1 = n$. Then for any $\psi \in E_{\lambda_1}^\perp$ with $||\psi - Q_{\bar{g}}||_{C^\infty(S^n,\bar{g})}$ sufficiently small, there exists a metric $g$ near $\bar{g}$ such that $$Q_g = \psi.$$
\end{theorem}

%%%%%%%%%%%%%%%%%%%%%%%%%%%%%%%%%%%%%%%%%%%%

\section{Rigidity Phenomena of Flat manifolds}

Let $(M,g)$ be a closed Riemannian manifold. Consider a deformation of $(M,g)$, $g(t)=g+th$, $t \in (-\varepsilon, \varepsilon)$.\\

We have calculated the first variation of $Q$-curvature (see equation (\ref{Q_linearization})). In order to study the local rigidity of $Q$-curvature, we are going to calculate the second variation of $Q$-curvature.\\

First, we recall the following well known $2^{nd}$-variation formulae, which can be found in \cite{F-M}. For detailed calculations, we refer to the appendices of \cite{Yuan}.

\begin{proposition}
We have the following $2^{nd}$-variation formulae for metrics,
\begin{align}
g''_{ij} = \left.\frac{d^2}{dt^2}\right|_{t=0} (g(t))_{ij} = 0,
\end{align}
and
\begin{align}
g''^{ij} = \left.\frac{d^2}{dt^2}\right|_{t=0} (g(t))^{ij} = 2 h_k^j h^{ik}.
\end{align}

Also for Christoffel Symbols,
\begin{align}
{\Gamma''}_{ij}^k = \left.\frac{d^2}{dt^2}\right|_{t=0} \Gamma(g(t))^k_{ij} = - h^{kl} \left( \nabla_i  h_{jl} + \nabla_j h_{il}
- \nabla_l h_{ij} \right).
\end{align}
\end{proposition}

\begin{lemma}
The second variation of $Q$-curvature is
\begin{align}
\left.\frac{d^2}{dt^2}\right|_{t=0} Q(g(t)) =& A_n [ \Delta_g R'' + 2 \nabla^2 R \cdot ( h \times h ) - 2 h \cdot \nabla^2 R' + ( 2 \delta h + d (tr h) ) \cdot dR' \\ \notag &\ \ \ \ \ \ +  h^{ij} ( 2 \nabla_i h_j^k - \nabla^k h_{ij} ) \nabla_k R - h \cdot ((2 \delta h + d(tr h))  \otimes d R )  ]\\ \notag
& + B_n [ 4 (Ric \times Ric) \cdot ( h \times h ) + 2 | Ric \times h |^2 - 8 Ric' \cdot ( Ric \times h )\\ \notag &\ \ \ \ \ \ + 2 Ric''\cdot Ric + 2 | Ric' |^2  ]\\ \notag
& + C_n [ 2 R R'' + 2(R')^2], \notag
\end{align}
where $Ric'$, $Ric''$; $R'$, $R''$ are the first and second variations of Ricci tensor and scalar curvature at $g$ respectively.
\end{lemma}

\begin{proof}
Choose normal coordinates at any point $p \in M$, \emph{i.e.} $\Gamma_{jk}^i = 0$ at $p$, then
\begin{align*}
&\left.\frac{d^2}{dt^2}\right|_{t=0}(\Delta_{g(t)}R(g(t)))\\
=& (g^{ij}(\partial_i\partial_j R - \Gamma_{ij}^k \partial_k R))''\\
=& g''^{ij} \partial_i\partial_j R  + 2 g'^{ij}(\partial_i\partial_j R' - {\Gamma'}_{ij}^k \partial_k R) + g^{ij}(\partial_i\partial_j R'' - {\Gamma''}_{ij}^k \partial_k R - 2{\Gamma'}_{ij}^k \partial_k R')\\
=& 2 h_k^j h^{ik} \nabla_i \nabla_j R - 2 h^{ij} (\nabla_i\nabla_j R' - {\Gamma'}_{ij}^k \nabla_k R) + \Delta R'' - g^{ij} {\Gamma''}_{ij}^k \nabla_k R - 2 g^{ij} {\Gamma'}_{ij}^k \nabla_k R'\\
=& \Delta R'' + 2 \nabla^2 R \cdot (h \times h) - 2 h\cdot \nabla^2 R' + ( 2 \delta h + d trh) \cdot dR' +  h^{ij} ( 2 \nabla_i h_j^k - \nabla^k h_{ij}) \nabla_k R\\
 &- h \cdot ((2 \delta h + d tr h) \otimes d R),
\end{align*}
by substituting the expression of ${\Gamma'}_{ij}^k$ and ${\Gamma''}_{ij}^k$.\\

And
\begin{align*}
&\left.\frac{d^2}{dt^2}\right|_{t=0} |Ric(g(t))|^2_{g(t)}\\
=& \left( g^{ik} g^{jl} R_{ij} R_{kl} \right)''\\
=& 2 g''^{ik}R_{ij} R_k^j + 2 g'^{ik} g'^{jl} R_{ij} R_{kl} + 8 g'^{ik} R'_{ij} R_k^j + 2 R''_{ij} R^{ij} + 2 g^{ik} g^{jl} R'_{ij} R'_{kl}\\
=& 4 (Ric \times Ric) \cdot ( h \times h) + 2 |Ric \times h|^2 - 8 Ric' \cdot (Ric \times h) + 2 Ric'' \cdot Ric + 2|Ric'|^2. \\
\end{align*}

Also,
\begin{align*}
\left.\frac{d^2}{dt^2}\right|_{t=0} R(g(t))^2 = (R^2)'' = 2 (R')^2 + 2 R R''.
\end{align*}

We prove the lemma by combining all three parts together.
\end{proof}

Simply by taking $Ric = 0$ and $R = 0$, we get the second variation for $Q$-curvature at a Ricci flat metric.
\begin{corollary}\label{2nd_variation_Ricci_flat}
Suppose the metric $g$ is Ricci flat, then
\begin{align}
D^2 Q_g \cdot ( h, h )
=A_n(\Delta_{g}R'' - 2 h \cdot \nabla^2 R' + ( 2 \delta h + d (tr h)) \cdot dR') + 2 B_n |Ric'|^2 + 2C_n(R')^2.
\end{align}
\end{corollary}

Now we assume $(M,\bar{g}, f)$ is a $Q$-singular metric.\\

Consider the functional
\begin{align*}
\mathscr{F}(g)=\int_M Q_g \cdot f dv_{\bar{g}}.
\end{align*}
Note that here we fix the volume form to be the one associated to the $Q$-singular metric $\bar{g}$.\\

\begin{remark}
The analogous functional
\begin{align*}
\mathscr{G}(g)=\int_M R_g \cdot f dv_{\bar{g}},
\end{align*}
plays a fundamental role in studying rigidity phenomena of vacuum static spaces (c.f. \cite{F-M, B-M, Q-Y_1}, \emph{etc.}).
\end{remark}

\begin{lemma}
The metric $\bar{g}$ is a critical point of the functional $\mathscr{F}(g)$.
\end{lemma}

\begin{proof}
For any symmetric 2-tensor $h \in S_2$, let $g(t) = \bar{g} + t h$, $t \in (-\varepsilon, \varepsilon)$ be a family of metrics on $M$. clearly, $g(0) = \bar{g}$ and $g'(0) = h$. Then
\begin{align*}
\left. \frac{d}{dt}\right|_{t=0} \mathscr{F}(g(t)) = \int_M DQ_{\bar{g}} \cdot h f dv_{\bar{g}} = \int_M \Gamma_{\bar{g}} h \cdot f dv_{\bar{g}} = \int_M h \cdot \Gamma_{\bar{g}}^* f \ dv_{\bar{g}} = 0,
\end{align*}
\emph{i.e.} $\bar{g}$ is a critical point for the functional $\mathscr{F}(g)$.
\end{proof}

Furthermore, if we assume $\bar{g}$ is a flat metric, by Theorem \ref{Static_Ricci_flat}, we can take $f$ to be a nonzero constant. In particular, we can take $f \equiv 1$, since $Q$-singular equation is a linear equations system with respect to $f$.\\

\begin{lemma}\label{2nd_variation_flat}
Let $\bar{g}$ be a flat metric and $f \equiv 1$. Suppose $\delta h = 0$, then the second variation of $\mathscr{F}$ at $\bar{g}$ is given by
\begin{align}
D^2 \mathscr{F}_{\bar{g}} \cdot (h,h) = -2 \alpha_n \int_M|\Delta(tr h)|^2 dv_{\bar{g}} + \frac{1}{2}B_n\int_M(|\Delta \overset{\circ}{h}|^2)dv_{\bar{g}},
\end{align}
where $\overset{\circ}{h}$ is the traceless part of $h$ and $\alpha_n := - \frac{1}{2}\left(A_n + \frac{n+1}{2n}B_n + 2C_n \right)= \frac{(n^2 - 2) (n^2 - 2n - 2)}{8n(n-1)^2(n-2)^2} > 0$, $B_n=-\frac{2}{(n-2)^2}< 0$, for any $n \geq 3$.
\end{lemma}

\begin{proof}

With the aid of Corollary \ref{2nd_variation_Ricci_flat}, we have
\begin{align*}
&\left. \frac{d^2}{dt^2}\right|_{t=0} \mathscr{F}(g(t))\\
 =& \int_M Q''_{\bar{g}} dv_{\bar{g}}\\
 =& \int_M \left[A_n(\Delta R'' - 2 h \cdot \nabla^2 R' + ( 2 \delta h + d (tr h)) \cdot dR') + 2 B_n |Ric'|^2 + 2C_n(R')^2 \right] dv_{\bar{g}}\\
 =& \int_M \left[A_n(- 2 \delta h \cdot d R' + ( 2 \delta h + d (tr h)) \cdot dR') + 2 B_n |Ric'|^2 + 2C_n(R')^2 \right] dv_{\bar{g}}\\
 =& \int_M \left[A_n(   - \Delta (tr h)) \cdot R' + 2 B_n |Ric'|^2 + 2C_n(R')^2 \right] dv_{\bar{g}}\\
 =& \int_M \left[A_n(   - \Delta (tr h)) \cdot (- \Delta (tr h) + \delta^2 h - Ric \cdot h) + 2 B_n |Ric'|^2 + 2C_n(R')^2 \right] dv_{\bar{g}}\\
 =& \int_M \left[A_n( \Delta (tr h))^2 + 2 B_n |Ric'|^2 + 2C_n(R')^2 \right] dv_{\bar{g}},
\end{align*}
where the last step is due to the assumptions of $\bar{g}$ being a flat metric and divergence-free property of $h$.\\

For exact the same reasons, by Proposition \ref{1st_variation_Ricci_scalar}, we have
\begin{align*}
|Ric'|^2 = \frac{1}{4} | \Delta h  + \nabla^2 (tr h)|^2 = \frac{1}{4} \left( |\Delta h|^2 + 2 \Delta h \cdot \nabla^2 (trh) + |\nabla^2 (tr h)|^2 \right)
\end{align*}
and
\begin{align*}
(R')^2 = ( \Delta (tr h))^2.
\end{align*}

Thus
\begin{align*}
  \int_M |Ric'|^2 dv_{\bar{g}}=& \frac{1}{4} \int_M\left[\left( |\Delta h|^2 + 2 \Delta h \cdot \nabla^2 (trh) + |\nabla^2 (tr h)|^2 \right)\right] dv_{\bar{g}}\\
 =& \frac{1}{4}\int_M \left[ \left( |\Delta h|^2 + 2 \delta\Delta h \cdot d(trh) + \delta\nabla^2 (tr h) \cdot d(tr h) \right)\right] dv_{\bar{g}}\\
 =& \frac{1}{4}\int_M \left[\left( |\Delta h|^2 + 2 \Delta \delta h \cdot d(trh) + \delta d (tr h) \cdot \delta d(tr h) \right)\right] dv_{\bar{g}}\\
 =& \frac{1}{4}\int_M \left[\left( |\Delta h|^2  + |\Delta (tr h)|^2 \right)\right] dv_{\bar{g}}\\
 =& \frac{1}{4}\int_M \left[ |\Delta \overset{\circ}{h}|^2 + \frac{n+1}{n}( \Delta (tr h))^2  \right] dv_{\bar{g}}.
\end{align*}

Now
\begin{align*}
\left. \frac{d^2}{dt^2}\right|_{t=0} \mathscr{F}(g(t)) =& \int_M \left[(A_n + \frac{n+1}{2n}B_n + 2 C_n)( \Delta (tr h))^2 +\frac{1}{2}B_n |\Delta \overset{\circ}{h}|^2 \right] dv_{\bar{g}}\\
 =& -2 \alpha_n \int_M|\Delta(tr h)|^2 dv_{\bar{g}} + \frac{1}{2}B_n\int_M(|\Delta \overset{\circ}{h}|^2)dv_{\bar{g}}.
\end{align*}

This gives the equation we claimed.
\end{proof}

Now we are ready to prove Theorem \ref{flat_local_rigidity}.

\begin{theorem}%\label{flat_local_rigidity}
For $n \geq 3$, let $(M^n,\bar{g})$ be a closed flat Riemannian manifold and $g$ be a metric on $M$ with $$Q_g \geq 0.$$ Suppose $||g - \bar{g}||_{C^2(M, \bar{g})}$ is sufficiently small, then $g$ is also flat.
\end{theorem}

\begin{proof}%[Proof of Theorem \ref{flat_local_rigidity}]

Since $g$ is $C^2$-closed to $\bar{g}$, by the \emph{Slice Theorem} (Theorem \ref{Slice_Theorem}), there exists a diffeomorphism $\varphi \in \mathscr{D}( M ) $ such that $$h := \varphi^* g - \bar{g}$$ is divergence free with respect to $\bar{g}$ and $$||h||_{C^2(M,\bar{g})} \leq N ||g-\bar{g}||_{C^2(M,\bar{g})},$$ where $N > 0$ is a constant only depends on $(M,\bar{g})$.\\

Apply Lemma \ref{2nd_variation_flat}, we can expand $\mathscr{F}(\varphi^*g)$ at $\bar{g}$ as follows,
\begin{align}\label{eqn:F_expansion}
\mathscr{F}(\varphi^*g) &= \mathscr{F}(\bar g) + D\mathscr{F}_{\bar g} \cdot h + \frac{1}{2}D^2 \mathscr{F}_{\bar g} \cdot (h, h) + E_3\\
\notag   &=   - \alpha_n \int_M|\Delta(tr h)|^2 dv_{\bar{g}} + \frac{1}{4}B_n\int_M |\Delta \overset{\circ}{h}|^2 dv_{\bar{g}} + E_3,
\end{align}
where $$|E_3| \leq C_0 \int_M |h|\ |\nabla^2 h|^2 \ dv_{\bar{g}}$$ for some constant $C_0=C_0 (n, M, \bar g)> 0$.\\

We know $$\mathscr{F}(\varphi^*g) = \int_M Q_g\circ\varphi\ dv_{\bar{g}} \geq 0$$ since $Q_{g}\geq0$ and $\varphi$ is a diffeomorphism near the identity. Since
$$\alpha_n =-\frac{1}{2}\left(A_n + \frac{n+1}{2n}B_n + 2C_n \right) > 0,\ \ B_n = - \frac{2}{(n-2)^2} < 0$$ for $n\geq 3$, we can take 
$\mu_n > 0$, a sufficiently small constant only depends on the dimension $n$ such that
$$\min\{ \alpha_n - \frac{\mu_n}{n} , - \frac{1}{4}B_n- \mu_n\} > 0.$$

Therefore, with the aid of equation (\ref{eqn:F_expansion}), we have
\begin{align*}
&\mu_n \int_M |\Delta h|^2 dv_{\bar{g}}\\
\leq& \left( \alpha_n - \frac{\mu_n}{n}\right)\int_M|\Delta(tr h)|^2 dv_{\bar{g}} + \left( - \frac{1}{4}B_n- \mu_n \right)\int_M |\Delta \overset{\circ}{h}|^2 dv_{\bar{g}} + \mu_n \int_M |\Delta h|^2 dv_{\bar{g}}\\
=&  \alpha_n \int_M|\Delta(tr h)|^2 dv_{\bar{g}} - \frac{1}{4}B_n\int_M |\Delta \overset{\circ}{h}|^2 dv_{\bar{g}}\\
=& - \mathscr{F}(\varphi^*g) + E_3\\
\leq& |E_3|\\
\leq& C_0 \int_M |h|\ |\nabla^2{h}|^2 \ dv_{\bar{g}}.
\end{align*}

Suppose $g$ is sufficiently $C^2$-closed to $\bar g$, say $||g-\bar{g}||_{C^2(M,\bar{g})} < \frac{\mu_n}{2NC_0}$, then $$||h||_{C^0(M,\bar{g})} \leq ||h||_{C^2(M,\bar{g})} \leq N ||g-\bar{g}||_{C^2(M,\bar{g})} < \frac{\mu_n}{2C_0}$$ and therefore,
\begin{align*}
\mu_n\int_M |\nabla^2 h|^2dv_{\bar{g}} = \mu_n\int_M |\Delta h|^2dv_{\bar{g}} \leq C_0 \int_M |h|\ |\nabla^2{h}|^2 \ dv_{\bar{g}} \leq \frac{\mu_n}{2}\int_M |\nabla^2 h|^2dv_{\bar{g}},
\end{align*}
which implies $\nabla^2 h = 0$ on $M$. \\

Now we have
\begin{align*}
\int_M |\nabla h |^2 dv_{\bar{g}} = - \int_M h \Delta h\  dv_{\bar{g}} = 0.
\end{align*}
That is $\nabla h = 0$.\\

Since $\bar{g}$ is flat, then on a neighborhood $U_p$ for any $ p \in M$, we can find a local coordinates, such that $\bar{g}_{ij} = \delta_{ij}$ and $\partial_k \bar{g}_{ij} = 0$, $i,j,k = 1, \cdots, n$ on $U_p$.\\

Under the same coordinates, Christoffel symbols of $\varphi^*g$ are
\begin{align*}
\Gamma^k_{ij}(\varphi^*g) &= \frac{1}{2} (\varphi^*g)^{kl} \left(\partial_i (\bar{g}_{jl} + h_{jl}) + \partial_j (\bar{g}_{il} + h_{il}) - \partial_l (\bar{g}_{ij} + h_{ij}) \right) \\
&= \frac{1}{2} (\varphi^*g)^{kl} (\nabla_i  h_{jl} + \nabla_j h_{il} - \nabla_l h_{ij}) = 0
\end{align*}
on $U_p$, since $h$ is parallel with respect to $\bar{g}$. \\

Thus the Riemann curvature tensor of $\varphi^* g$ vanishes identically on $U_p$ for any $p \in M$, which implies that the metric $\varphi^*g$ is flat and so is $g$.
\end{proof}

\begin{remark}
Fischer and Marsden proved an analogous result for the scalar curvature. (see \cite{F-M}) 
\end{remark}

As an application, we can get the local rigidity of $Q$-curvature on compact domain of $\mathbb{R}^n$, which can be thought as an analogue of rigidity  part of \emph{Positive Mass Theorem}.

\begin{corollary}\label{flat_domain_rigidity}
Suppose $\Omega \subset \mathbb{R}^n$ is a compactly contained domain. Let $\delta$ be the flat metric and $g$ be a metric on $\mathbb{R}^n$ satisfying
\begin{itemize}
\item $Q_g \geq 0$,
\item $supp(g - \delta) \subset \Omega$,
\item $||g - \delta||_{C^2(\mathbb{R}^n, \delta)}$ is sufficiently small;
\end{itemize}
then $g$ is also flat.
\end{corollary}

\begin{proof}

Since $\Omega$ is compactly contained in $\mathbb{R}^n$, we can choose a rectangle domain $\Omega'$ which contains $\Omega$ strictly. Thus, the metric $g$ is identically the same as the Euclidean metric on $\Omega' - \Omega$. Now we can derive a metric with nonnegative $Q$-curvature on the torus $T^n$ by identifying the boundary of $\Omega'$ correspondingly. Clearly, this new metric on $T^n$ is $C^2$-closed to the flat metric, hence has to be flat by Theorem \ref{flat_local_rigidity}. Now the claim follows.
 
\end{proof}

It would be interesting to ask whether there is a global rigidity result for $Q$-curvature. No result is known so far to the best of authors' knowledge, but we observed that the global rigidity holds in some special cases.

\begin{proposition}\label{prop:conf_Ricci_flat}
Let $(M, g)$ be a closed Riemannian manifold with $$Q_g \geq 0$$ pointwisely. Suppose that $g$ is conformal to a Ricci flat metric, then $(M, g)$ is Ricci flat.
\end{proposition}

\begin{proof}
For $n = 4$, by the assumptions, there exists a smooth function $u$ such that $g = e^{2u} \bar g$, where $\bar g$ is a Ricci flat metric on $M$. By (\ref{eqn:conf_Q_4}) in the appendix, we have 
$$Q_g = e^{-4u} ( P_{\bar g} u + Q_{\bar g} )= e^{-4u} \Delta_{\bar g}^2 u \geq 0.$$ Hence $ \Delta_{\bar g}u$ is subharmonic on $M$ which implies that $u$ is a constant. Therefore, $g$ is also Ricci flat. \\

Similarly, for $n \neq 4$, there exists a positive function $u > 0$ such that $g = u^{\frac{4}{n-4}}\bar{g}$, where $\bar g$ is a Ricci flat metric on $M$. Then by (\ref{eqn:conf_Q_n}), $$Q_g = \frac{2}{n-4}u^{-\frac{n+4}{n-4}} P_{\bar{g}} u =   \frac{2}{n-4} u^{-\frac{n+4}{n-4}} \Delta_{\bar{g}}^2 u. $$

For $n=3$, we have $$\Delta_{\bar{g}}^2 u \leq 0,$$ and for $n>4$, 
$$\Delta_{\bar{g}}^2 u \geq 0,$$ which imply that $u$ is a constant in both cases and thus $g$ is Ricci flat.

\end{proof}

In particular, we can consider tori $T^n$. First, we need a lemma which characterizes the conformally flat structure on $T^n$.

\begin{lemma}\label{lem:conf_structure_tori}
On the torus $T^n$, any locally conformally flat metric has to be conformal to a flat metric.
\end{lemma}

\begin{proof}
Let $g$ be a locally conformally flat Riemannian metric on $T^n$. According to the solution of Yamabe problem, $g$ is conformal to a metric $\bar g$, whose scalar curvature is a constant. 

Suppose $R_{\bar g} < 0$, by Proposition 1.2 in \cite{S-Y_3}, the fundamental group of $T^n$ is non-amenable. But this contradicts to the fact that $\pi_1(T^n)$ is abelian, which is amenable. Thus, $R_{\bar g} \geq 0$, which implies that $\bar g$ is flat by famous results of Schoen-Yau and Gromov-Lawson (\cite{S-Y_1, S-Y_2, G-L_1, G-L_2}).
\end{proof}

Now we can derive the rigidity of tori with respect to nonnegative $Q$-curvature (Theorem \ref{thm:rigidity_tori}):

\begin{theorem}%\label{thm:rigidity_tori}
Suppose $g$ is a locally conformally flat metric on $T^n$ with $$Q_g \geq 0,$$ then $g$ is flat. 

In particular, any metric $g$ on $T^4$ with nonnegative $Q$-curvature has to be flat.
\end{theorem}
%%%%statement thm 1.10
\begin{proof}%[Proof of Theorem\ref{thm:rigidity_tori}]
By Lemma \ref{lem:conf_structure_tori}, we can see that $g$ is conformal to a flat metric. Applying Proposition \ref{prop:conf_Ricci_flat}, we conclude that $g$ has to be flat.

In particular for $T^4$, we have the \emph{Gauss-Bonnet-Chern formula}, 
$$\int_{T^4} \left( Q_g + \frac{1}{4}|W_g|^2 \right) dv_{g} = 8\pi^2 \chi(T^4) = 0$$
on $T^4$. Thus the non-negativity of $Q$-curvature automatically implies that Weyl tensor $W_g$ vanishes identically on $T^4$, which means $g$ is locally conformally flat. Therefore, $g$ is flat by the previous argument. 
\end{proof}

As for dimension 3, we have the following result.

\begin{proposition}
The $3$-dimensional torus $T^3$ does not admit a metric with constant scalar curvature and nonnegative $Q$-curvature, unless it is flat.
\end{proposition}

\begin{proof}

Suppose such a metric $g$ exists, its scalar curvature is non-positive and it is flat only if $R_g = 0$ (c.f. \cite{ S-Y_1, S-Y_2, G-L_1, G-L_2}). 

If it is non-flat, without loss of generality, we can assume $$R_g = -1.$$ Then we have $$Q_g = - \frac{1}{4}\Delta_g R_g - 2 |Ric_g|^2 + \frac{23}{32} R_g^2 = - 2 |Ric_g|^2 + \frac{23}{32} R_g^2 \geq 0.$$

That is $$|Ric_g|^2 \leq \frac{23}{64}R_g^2 = \frac{23}{64}.$$

At any point $p \in M$, choose an orthonormal basis $\{e_1, e_2, e_3 \}$ for $T_p M$ such that the Ricci tensor at $p$ is diagonal. Let $\lambda_i$, $i = 1,2,3$ be the eigenvalues of $Ric_g (p)$. Then we have $$\lambda_1 + \lambda_2 + \lambda_3 = -1$$ and $$\lambda_1^2 + \lambda_2^2 + \lambda_3^2 \leq \frac{23}{64}.$$

Hence for $i \neq j$, we have
\begin{align*}
0 \geq& \lambda_i^2 + \lambda_j^2 + ( 1+ (\lambda_i + \lambda_j) )^2 - \frac{23}{64} \\
=& (\lambda_i + \lambda_j)^2 + 2 (\lambda_i + \lambda_j) + \lambda_i^2 + \lambda_j^2 + \frac{41}{64} \\
\geq &\frac{3}{2}(\lambda_i + \lambda_j)^2 + 2(\lambda_i + \lambda_j) + \frac{41}{64},
\end{align*}
where the last step was achieved by the mean value inequality $$\lambda_i^2 + \lambda_j^2 \geq \frac{(\lambda_i + \lambda_j)^2}{2} .$$
That is,
$$(\lambda_i + \lambda_j)^2 + \frac{4}{3}(\lambda_i + \lambda_j) + \frac{41}{96} \leq 0,$$
which implies that $\lambda_i + \lambda_j < - \frac{2}{3} + \frac{\sqrt{10}}{24}  < -\frac{1}{2}$.

Then the sectional curvature of the plane spanned by $e_i$ and $e_j$ satisfies
\begin{align*}
K_{ij} = R_{ijji} = R_{ii} g_{jj} + R_{jj} g_{ii} - \frac{1}{2} R g_{ii}g_{jj} = \frac{1}{2} + (\lambda_i + \lambda_j) < 0.
\end{align*}

Thus $g$ has negative sectional curvature. But the torus does not admit a metric with nonpositive sectional curvature by Corollary 2 in \cite{B-B-E}, which is a contradiction.
\end{proof}

\begin{remark}
From the last inequality in the above argument, we can easily see that the conclusion actually allows a perturbation on the metric. That is, any metric on $T^3$ with scalar curvature sufficiently $C^4$-closed to a negative constant can not have a nonnegative $Q$-curvature, unless it is flat. 
\end{remark}

\begin{remark}
According to the solution of Yamabe problem, there is a metric with constant scalar curvature in each conformal class on any closed Riemannian manifold. Thus we can conclude that for any metric $g$ on $T^n$, if $g$ is not conformally flat, then it must be conformal to a metric with constant negative scalar curvature. However, $Q$-curvature cannot stay nonnegative when performing conformal changes on metrics in general. We hope to develop some techniques to resolve this issue in the future.
\end{remark}

According to the \emph{Bieberbach Theorem} (See Theorem 10.33 in \cite{C-L-N}), any $n$-dimensional flat Riemannian manifold is a finite quotient of the torus $T^n$. So based on Theorem \ref{flat_local_rigidity}, \ref{thm:rigidity_tori} and above observations, we give the following conjecture:

\begin{conjecture}\label{Positive_Q_on_torus}
For $n \geq 3$, let $(M^n,\bar{g})$ be a connected closed flat Riemannian manifold, then there is no metric with pointwisely positive $Q$-curvature. Moreover, if $g$ is a metric on $M$ with $$Q_g\geq 0,$$ then $g$ is flat.
\end{conjecture}

%If this conjecture is true, then we will have the following conclusion immediately.

%\begin{corollary}[Corollary of Conjecture \ref{Positive_Q_on_torus}]
%For $n\geq 3$, the $n$-dimensional torus $T^n$ cannot carry a metric with positive $Q$-curvature. Moreover, if a metric $g$ on $T^n$ satisfies that $$Q_g \geq 0,$$ then it is flat necessarily.
%\end{corollary}

\begin{remark}
This is a higher order analogue of the famous result due to Schoen-Yau and Gromov-Lawson (\cite{S-Y_1, S-Y_2, G-L_1, G-L_2}), which says that there is no nonnegative scalar curvature metric on tori unless they are flat. \end{remark}

\section{Relation between vacuum static spaces and $Q$-singular spaces}

In this section, we will discuss the relation between $Q$-singular spaces and vacuum static spaces.

\begin{definition}%[\cite{Q-Y}]
We say a complete Riemannian manifold $(M,g)$ is vacuum static, if there is a smooth function $f (\not\equiv 0)$ on $M$ solving the following \emph{vacuum static equation} 
\begin{align}
\gamma_g^* f := \nabla^2 f - \left(Ric_g - \frac{R_g}{n-1} g \right) f = 0.
\end{align}
We also refer $(M,g,f)$ as a \emph{vacuum static space}, if $f (\not\equiv 0) \in \ker \gamma_g^*$.
\end{definition}

Now we can prove the main result (Theorem \ref{Q-static_R-static}) in this section.

\begin{theorem}%\label{Q-static_R-static}
Let $\mathcal{M}_R$ be the space of all closed vacuum static spaces and $\mathcal{M}_Q$ be the space of all closed $Q$-singular spaces. 
Suppose $(M, g, f) \in \mathcal{M}_R \cap \mathcal{M}_Q$, then $M$ is Einstein necessarily. In particular, it has to be either Ricci flat or isometric to a round sphere.
\end{theorem}
%%%%statement of Thm1.11
\begin{proof}%[Proof of Theorem \ref{Q-static_R-static}]
If $(M,g,f)$ is vacuum static, $M$ has constant scalar curvature necessarily (c.f. \cite{F-M}). Then the $Q$-singular equation can be reduced to
\begin{align}\label{eqQ_and_Rstatic}
\Gamma_g^* f =& A_n \left( - g \Delta^2 f + \nabla^2
\Delta f - Ric \Delta f \right)- 2 C_n R \left( g\Delta f - \nabla^2 f + f
Ric \right)\\
&\notag - B_n \left( \Delta (f Ric) + 2 f
\overset{\circ}{Rm}\cdot Ric + g \delta^2 (f Ric) + 2 \nabla \delta (f
Ric) \right)\\ =&
\notag 0.
\end{align}

By \emph{Contracted $2^{nd}$ Bianchi Identity}
\begin{align*}
\delta Ric = - \frac{1}{2} dR = 0,
\end{align*}
we can simplify (\ref{eqQ_and_Rstatic}) furthermore,
\begin{align*}
&-\frac{1}{B_n}\Gamma_g^* f \\
=& -\frac{A_n}{B_n} \left( \nabla^2 \Delta f - g \Delta^2 f - Ric \Delta f \right) - \frac{2C_n}{B_n}R \left( \nabla^2 f - g \Delta f - Ric f \right) \\ &+ f (\Delta Ric + 2 \overset\circ{Rm} \cdot Ric) + g (Ric \cdot \nabla^2 f) + Ric \Delta f - 2 \nabla^2 f \times Ric + 2 C \cdot \nabla f\\
%=& -\frac{A_n}{B_n}\ \gamma_g^* (\Delta f) + \frac{2C_n}{B_n}R\ \gamma_g^* f + f (\Delta Ric + 2 \overset\circ{Rm} \cdot Ric) + g (Ric \cdot \nabla^2 f)  \\
%&+ Ric \Delta f - 2 \nabla^2 f \times Ric + 2 C \cdot \nabla f\\
%=& -\frac{A_n}{B_n}\ \gamma_g^* (\Delta f) + \frac{2C_n}{B_n}R\ \gamma_g^* f + f \Delta_L Ric  + g (Ric \cdot \nabla^2 f)  + Ric \Delta f - 2 \left( \nabla^2 f - f Ric \right)\times Ric + 2 C \cdot \nabla f\\
=& -\frac{A_n}{B_n}\ \gamma_g^* (\Delta f) - \frac{2C_n}{B_n}R\ \gamma_g^* f + f \Delta_L Ric + g (Ric \cdot \nabla^2 f)  + Ric \Delta f  +\frac{2 R }{n-1}f Ric + 2 C \cdot \nabla f\\
=& 0,
\end{align*}
where $\Delta_L$ is the Lichnerowicz Laplacian and $C$ is Cotton tensor,
\begin{align*}
C_{ijk} = \left( \nabla_i R_{jk} - \frac{1}{2(n-1)}g_{jk}\nabla_i R\right) - \left( \nabla_j R_{ik} - \frac{1}{2(n-1)}g_{ik} \nabla_j R \right) = \nabla_i R_{jk} - \nabla_j R_{ik}
\end{align*}
and
\begin{align*}
(C \cdot \nabla f)_{jk} := C_{ijk} \nabla^i f.
\end{align*}

Now suppose $g$ is also vacuum static, then $\Delta f = -\frac{R}{n-1}f$ and hence
\begin{align*}
\gamma_g^* (\Delta f) = -\frac{R}{n-1} \gamma_g^* f = 0.
\end{align*}

Thus we have
\begin{align*}
\Gamma_g^* f
= f \Delta_L Ric + g (Ric \cdot \nabla^2 f)  + \frac{R}{n-1} f Ric + 2 C \cdot \nabla f = 0.
\end{align*}

Taking trace and applying the vacuum static equation,
\begin{align*}
tr\ \Gamma_g^* f
=& f \Delta R + n Ric \cdot \nabla^2 f + \frac{R^2}{n-1} f\\
=& n Ric \cdot \left( Ric - \frac{R}{n-1}g\right) f + \frac{R^2}{n-1} f\\
=& n \left( |Ric|^2 - \frac{R^2}{n} \right) f \\
=& n \left|Ric - \frac{R}{n} g\right|^2 f\\
=& 0.
\end{align*}

Since in an vacuum static space, $df \neq 0$ on $f^{-1}(0)$ (c.f. \cite{F-M}). Then $f^{-1}(0)$ is a regular hypersurface in $M$ and hence $f \neq 0$ on a dense subset of $M$, which implies that $$Ric = \frac{R}{n} g.$$ \emph{i.e.} $(M,g)$ is Einstein.

For a closed vacuum static space, its scalar curvature is nonnegative necessarily. Hence the assertion follows easily from Theorem \ref{Classificaition_Q_singular_Einstein}.

\end{proof}

%%%%%%%%%%%%%%%%%%%%%%%%%%
\appendix
\section{Conformal covariance of $Q$-curvature and Paneitz operator}

In this appendix, we will give a brief introduction to $Q$-curvature and Paneitz operator from the viewpoint of conformal geometry. \\

Let $(M^2, g)$ be a $2$-dimensional Riemannian surface. Consider the conformal metric $\tilde{g} = e^{2u}g$. A simple calculation shows that the \emph{Laplacian-Beltrami operator} is in fact conformally covariant: 
\begin{equation}
\Delta_{\tilde{g}} = e^{-2u}\Delta_{g}.
\end{equation}
And Gaussian curvature satisfies 
\begin{equation}
K_{\tilde{g}} = e^{-2u}\left(-\Delta_{g}u + K_{g} \right).
\end{equation}\\

For $n\geq 3$, a second order differential operator called \emph{conformal Laplacian} associated to a  Riemannian metric $g$ is defined as
\begin{align}
L_{g}:=-\frac{4(n-1)}{(n-2)}\Delta_{g} + R_{g}.
\end{align}
If we take $\tilde{g}=u^{\frac{4}{n-2}}g$, $u>0$, to be a metric conformal to $g$, then 
\begin{equation}
L_{\tilde{g}}\phi = u^{-\frac{n+2}{n-2}}L_{g}(u\phi)
\end{equation}
for any $\phi \in C^2(M)$. Hence $L_g$ can be thought as a higher dimensional analogue of the \emph{Laplacian-Beltrami operator} on surfaces. By taking $\phi \equiv 1$, the scalar curvature can be calculated as follows:
\begin{equation}
R_{\tilde{g}} =u^{-\frac{n+2}{n-2}} L_{g}u.
\end{equation}\\

As a generalization, we seek for a higher order operator which satisfies the similar transformation law. In \cite{Paneitz}, Paneitz introduced a $4^{th}$-order differential operator $P_g$ (people call it \emph{Paneitz operator} now) for any dimension $n \geq 3$. Shortly after, Branson (\cite{Branson}) realized that the $0^{th}$-order term in Paneitz operator actually can be defined as a generalization of $Q$-curvature for dimensions $n\neq 4$:
\begin{align}
Q_{g} = A_n \Delta_{g} R_{g} + B_n |Ric_{g}|_{g}^2 + C_nR_{g}^2,
\end{align}
where $A_n = - \frac{1}{2(n-1)}$ , $B_n = - \frac{2}{(n-2)^2}$ and
$C_n = \frac{n^2(n-4) + 16 (n-1)}{8(n-1)^2(n-2)^2}$.

Thus Paneitz operator can be rewritten as 
\begin{align}
P_g = \Delta_g^2 - div_g \left[(a_n R_g g + b_n Ric_g) d\right] + \frac{n-4}{2}Q_g,
\end{align}
where $a_n = \frac{(n-2)^2 + 4}{2(n-1)(n-2)}$ and $b_n = - \frac{4}{n-2}$.

In fact, when $n=4$, let $\tilde{g} = e^{2u} g$ be a conformal metric, then the Paneitz operator follows similar transformation law as Laplacian-Beltrami operator on surfaces:
\begin{align}
P_{\tilde{g}} = e^{-4u}P_g
\end{align}
and $Q$-curvature satisfies
\begin{align}\label{eqn:conf_Q_4}
Q_{\tilde{g}} =  e^{-4u} \left(P_g u + Q_g \right).
\end{align}

As for $n\neq 4$, let $\tilde{g} = u^{\frac{4}{n-4}}g$, $u > 0$, then
\begin{align}
P_{\tilde{g}} \phi = u^{-\frac{n+4}{n-4}}P_g (u \phi),
\end{align}
for any $\phi \in C^4 (M)$. In particular, by taking $\phi \equiv 1$, we get the transformation law for $Q$-curvature:
\begin{align}\label{eqn:conf_Q_n}
Q_{\tilde{g}} = \frac{2}{n-4}u^{-\frac{n+4}{n-4}}P_g u.
\end{align}
In fact, Paneitz operator is a higher order analogue of conformal Laplacian.\\

Based on the similarity between $Q$-curvature and scalar curvature, one can consider a \emph{Yamabe-type problem} for $Q$-curvature: seeking for constant $Q$-curvature metrics in a given conformal class. Many people have contributed to this problem or related problems under different geometric assumptions.

In particular for closed $4$-manifolds, one of the significance for $Q$-curvature is that the total $Q$-curvature is a conformal invariant. Indeed, according to (\cite{Gur98, Gur99}), Paneitz operator $P_{g}$ and total $Q$-curvature together can provide us some geometric information about manifolds. Thus under some geometric conditions, the existence of constant $Q$-curvature metric can be obtained (c.f. \cite{CY95, DM08, LLL12}). 

As for dimensions other than four, different approaches were applied due to the lack of a maximum principle. One of the approaches is to understand the relation among $Q$-curvature equations, Green's function of Paneitz operator and certain geometric invariants. We refer to \cite{QR06, Lin15, GM14, HY14a, HY14b, GHL15} for readers who are interested in it.

%In addition, from the PDE aspect, authors in (\cite{DHL00, ER02}) also studied best Sobolev constants and investigated the optimal inequalities related to the fourth-order $Q$-curvature equation. 

\bibliographystyle{amsplain}

\end{document}